\documentclass[a4paper,10pt,fleqn]{amsart}

\usepackage{amssymb,amsmath,amsthm}
\usepackage{graphicx}
\usepackage{xypic}
\usepackage{color}
\usepackage{thmtools}
\usepackage{thm-restate}
\usepackage[shortlabels]{enumitem}
\usepackage{mathrsfs}
\usepackage[utf8]{inputenc}
\usepackage[T1]{fontenc}
\usepackage{cite}
\usepackage[hmarginratio={1:1},vmarginratio={1:1},textwidth=400pt,tmargin=1.2in,bmargin=1.2in]{geometry}

\newtheorem{theorem}{Theorem}[section]
\newtheorem{lemma}[theorem]{Lemma}
\newtheorem{proposition}[theorem]{Proposition}

\newtheorem{question}[theorem]{Question}
\newtheorem{corollary}[theorem]{Corollary}

\theoremstyle{definition}
\newtheorem{definition}[theorem]{Definition}

\theoremstyle{remark}
\newtheorem{remark}[theorem]{Remark}

\numberwithin{equation}{section}

\theoremstyle{plain}
\newtheorem*{question*}{Question}
\newtheorem*{maintheorem*}{Main Theorem}

\newcommand{\frakb}{\mathfrak{b}}
\newcommand{\frakc}{\mathfrak{c}}

\newcommand{\frakp}{\mathfrak{p}}

\newcommand{\frakx}{\mathfrak{x}}
\newcommand{\fraky}{\mathfrak{y}}

\newcommand{\frakwmlur}{\mathfrak{wmlur}}
\newcommand{\fraksc}{\mathfrak{sc}}
\newcommand{\frakskk}{\mathfrak{skk}}

\newcommand{\eps}{\varepsilon}

\newcommand{\R}{\mathbb{R}}

\newcommand{\PP}{\mathbb{P}}

\newcommand{\aA}{\mathcal{A}}
\newcommand{\bB}{\mathcal{B}}

\newcommand{\dD}{\mathcal{D}}

\newcommand{\fF}{\mathcal{F}}

\newcommand{\mM}{\mathcal{M}}
\newcommand{\nN}{\mathcal{N}}
\newcommand{\pP}{\mathcal{P}}

\newcommand{\xX}{\mathcal{X}}
\newcommand{\yY}{\mathcal{Y}}

\DeclareMathOperator{\add}{add}

\DeclareMathOperator{\dom}{dom}

\DeclareMathOperator{\non}{non}

\newcommand{\sm}{\setminus}
\newcommand{\sub}{\subseteq}
\DeclareMathOperator{\supp}{supp}
\DeclareMathOperator{\spn}{span}

\newcommand{\seq}[2]{\big\langle#1\colon\ #2\big\rangle}
\newcommand{\seqn}[1]{\big\langle#1\colon\ n\io\big\rangle}
\newcommand{\seqN}[1]{\big\langle#1\colon\ N\io\big\rangle}

\newcommand{\clopen}[1]{\left[#1\right]}

\newcommand{\ctblsub}[1]{\left[#1\right]^\omega}
\newcommand{\finsub}[1]{\left[#1\right]^{<\omega}}

\newcommand{\iA}{\in\aA}

\newcommand{\io}{\in\omega}

\newcommand{\wo}{{\wp(\omega)}}
\newcommand{\bo}{{\beta\omega}}

\newcommand{\cso}{\ctblsub{\omega}}
\newcommand{\fso}{\finsub{\omega}}

\newcommand{\elli}{{\ell_\infty}}

\newcommand{\noproof}{\hfill$\Box$}

\newcommand{\forces}{\Vdash}
\newcommand{\Sacks}{\mathbb{S}_\kappa}
\renewcommand{\SS}{\mathbb{S}}
\DeclareMathOperator{\lev}{lev}

\newcommand{\trnorm}[1]{{\left|\kern-0.17ex\left|\kern-0.17ex\left|#1\right|\kern-0.17ex\right|\kern-0.17ex\right|}}

\begin{document}

\title{A small Banach space $C(K)$ without nice renormings}

\author[T.\ Manev]{Todor Manev}
\address{Institute of Mathematics and Informatics, Bulgarian Academy of Sciences, Akad. G. Bonchev Str., Bl. 8, 1113 Sofia, Bulgaria}
\email{tmanev@math.bas.bg}

\author[D.\ Sobota]{Damian Sobota}
\address{Kurt G\"odel Research Center, Department of Mathematics, Vienna University, Kolingasse 14--16, 1090 Vienna, Austria.}
\email{damian.sobota@univie.ac.at}
\urladdr{https://www.logic.univie.ac.at/~dsobota/}

\author[L.\ Zdomskyy]{Lyubomyr Zdomskyy}
\address{Institut f\"ur Diskrete Mathematik und Geometrie, Technische Universit\"at Wien, Wiedner Hauptstra\ss e 8-10/104, 1040 Wien, Austria.}
\email{lzdomsky@gmail.com}
\urladdr{https://dmg.tuwien.ac.at/zdomskyy/}

\thanks{T. Manev was partially supported by the Bulgarian National Science Fund under Grant No.~KP-06-H92/6 (December 08, 2025). D. Sobota was supported by the Austrian Science Fund (FWF), grants ESP~108-N and 10.55776/PAT4720925. L. Zdomskyy was supported by the Austrian Science Fund (FWF), grants  
10.55776/I5930 and 10.55776/PAT5730424. The research presented in this paper was initiated during T. Manev's visit at the University of Vienna in November 2025, supported by the FWF grants ESP 108-N and 10.55776/I5930}

\begin{abstract}
We prove that consistently $\omega_1<\mathfrak{c}$ and there exists a compact space $K$ whose Banach space $C(K)$ of continuous real-valued functions is Grothendieck, has density $\omega_1$, and admits no renorming which is strictly convex or sequentially Kadets--Klee. 
\end{abstract}



\makeatletter
\@namedef{subjclassname@2020}{2020 Classification}
\makeatother
\renewcommand{\keywordsname}{Key words}

\maketitle

\section{Introduction\label{sec:intro}}

We start by recalling several standard definitions concerning norms on vector spaces. Let $(X,\|\cdot\|)$ be a normed space and put $S_X=\{x\in X\colon \|x\|=1\}$. We say that the norm $\|\cdot\|$ is:
\begin{itemize}[itemsep=1mm]
    \item \textit{strictly convex} or \textit{rotund} (in short, \textit{R}), if whenever $x,y\in S_X$ are such that $\big\|(x+y)/2\big\|=1$, then $x=y$,
    \item \textit{locally uniformly rotund} (\textit{LUR}) [resp. \textit{weakly locally uniformly rotund} (\textit{wLUR})], if whenever for $x\in S_X$ and $\seqn{y_n\in S_X}$ we have $\big\|\big(x+y_n\big)/2\big\|\to 1$, then $y_n\to x$ [resp., $y_n\xrightarrow{w}x$],
    \item \textit{midpoint locally uniformly rotund} (\textit{MLUR}) [resp. \textit{weakly midpoint locally uniformly rotund} (\textit{wMLUR})], if whenever for $x\in S_X$ and $\seqn{y_n\in X}$ we have $\big\|x\pm y_n\big\|\to 1$, then $y_n\to 0$ [resp., $y_n\xrightarrow{w} 0$], 
    \item \textit{Kadets--Klee} (\textit{KK}) [resp. \textit{sequentially Kadets--Klee} (\textit{SKK})], if the weak topology and the norm topology coincide on  $S_X$ [resp., if weakly convergent sequences in $S_X$ are norm convergent].
\end{itemize}

The relations between the above properties are presented in the following diagram of implications:
\vspace{-1cm}$$\xymatrixcolsep{3pc}
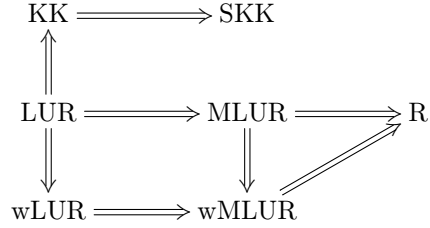
\xymatrix{
\\
\text{KK}\ar@{=>}[r]&   \text{SKK} & \\ 
\text{LUR}\ar@{=>}[r]\ar@{=>}[d]\ar@{=>}[u]&    \text{MLUR}\ar@{=>}[r]\ar@{=>}[d]&\text{R}\\
\text{wLUR}\ar@{=>}[r]&     \text{wMLUR}\ar@{=>}[ur]\\
}$$
\vspace{1mm}

None of the implications are reversible in general, but it is worth noting that by a result of Molt\'{o}, Orihuela, Troyanski, and Valdivia (see \cite[Theorem 227]{GMZ22}) a space with a wLUR norm admits a LUR renorming. Moreover, if a Banach space has both rotund and KK equivalent norms, then it has an equivalent LUR norm as well (Troyanski, see \cite[Theorems 407 and 410]{GMZ22}). For more basic information concerning the above-defined properties of norms, see the standard monographs \cite{DGZ93} and \cite{GMZ22}. 

In this paper we will investigate renorming properties of Banach spaces $C(K)$ of continuous real-valued functions on compact (Hausdorff) spaces $K$ endowed with the supremum norm. Those spaces have always constituted one of the central research lines in renorming theory, see e.g. 
\cite{Ale01,AB88,Hay90,Hay99,HJNR00,HMO07,HR90,Man26,ST10}. 
We are in particular interested in the issue how small those spaces can be in the case when they do \emph{not} admit any equivalent norms which are, e.g., strictly convex or LUR. The starting point for our research is the following question, attributed to A. Avil\'es (though we heard it also from W. Kubi\'s):

\begin{question*}\label{ques:main}
    Does there exist a Banach space of density $\omega_1$ which admits neither an equivalent strictly convex norm nor an equivalent sequentially Kadets--Klee norm?
\end{question*}


Recall that every separable Banach space admits an equivalent LUR norm and so a renorming which is also strictly convex and Kadets--Klee (see \cite{GMZ22}). Typical examples of Banach spaces without any strictly convex renormings are $\elli/c_0$ (Bourgain, see \cite[Theorem 803]{GMZ22}) and $\elli(\omega_1)$ (Day, see \cite[Theorem 800]{GMZ22}); the former space has density \textit{continuum}, $\frakc=2^\omega$, and the latter $2^{\omega_1}$. The space $\elli$, which has density $\frakc$, does not admit any equivalent wMLUR norm (Alexandrov and Babev \cite{AB88}), as well as no Kadets--Klee renorming (Troyanski, see \cite[Corollary 804]{GMZ22}), yet it has a strictly convex renorming (see \cite[Section 9.3.1]{GMZ22}). Note that all these spaces are $C(K)$-spaces in disguise: each space $\elli(\Gamma)$ is isometrically isomorphic to $C(\beta\Gamma)$, and $\elli/c_0$ is isometrically isomorphic to $C(\bo\sm\omega)$ (see Section \ref{sec:prelim}). Similar examples of Banach spaces $C(K)$ without renormings which are, e.g., wMLUR or SKK were given by Alexandrov \cite{Ale01}, though by the reasoning presented in \cite[Section 19.5]{KKLPS} all spaces studied by him are also of density at least $\frakc$.

Consistent partial positive answers to Question have been already known for some time, e.g. trivially under the assumption of the Continuum Hypothesis (\textsf{CH}; i.e. $\omega_1=\frakc$). A more sophisticated set-theoretic assumption implying an affirmative answer to the strictly convex part of the question is that of the existence of a \textit{Suslin line} $S$, that is, an unbounded dense complete linearly ordered topological space which satisfies the countable chain condition but is not separable; recall that such spaces exist, e.g., in models of \textsf{ZFC}+$\neg$\textsf{CH} satisfying also the Parametrized Diamond Principle $\Diamond(\non(\mM))$ such as the Miller model or the Sacks model, but do not exist if Martin's axiom (\textsf{MA}) and $\neg$\textsf{CH} hold (see \cite{Kun80,MHD03}). If $S$ is a Suslin line, then for the compact space $K_S$ obtained from $S$ by adding two end-points, by Haydon \textit{et al.} \cite{HJNR00}, the Banach space $C(K_S)$ admits no strictly convex renorming. Moreover, the space $K_S$ is always of weight $\omega_1$ and so $C(K_S)$ has density $\omega_1$, as required. Unfortunately, the space $C(K_S)$ does not fully answer Question as it does admit an equivalent Kadets--Klee norm. In fact, also by \cite{HJNR00}, this is the case for $C(K)$ whenever $K$ is a compact linearly ordered space.
	
For examples of small spaces without Kadets--Klee renormings, one can consider the following. Let $K$ be a non-metrizable scattered compact space such that $K^n$ is hereditarily separable for all $n \io$. By a classical result of Zenor (see \cite{zen80}) and the fact that scatteredness implies that the pointwise and weak topologies coincide on bounded subsets of $C(K)$, it follows that $C(K)$ is hereditarily Lindel\"of in the weak topology. If we suppose that for some equivalent norm $\trnorm{\,\cdot\,}$ on $C(K)$ the restrictions of the weak topology and the norm topology to the sphere $S_{(C(K),\trnorm{\,\cdot\,})}$ coincide, then this sphere is both metrizable and Lindel\"of, hence separable. Consequently, $C(K)$ is also separable and so $K$ is metrizable, which contradicts our initial assumption, proving that $C(K)$ cannot admit any equivalent KK norm $\trnorm{\,\cdot\,}$. Recall that classical consistent examples of non-metrizable scattered compact spaces with hereditarily separable finite powers include constructions by Kunen under $\mathsf{CH}$, by Fedorchuk under Jensen's Diamond Principle $\Diamond$, and crucially for our purposes, by Todor\v{c}evi\'c under the assumption that for the bounding number $\frakb$  we have $\frakb=\omega_1$ (see \cite[Theorem 2.4]{tod89}). The last example also satisfies $w(K) = \omega_1$ while the continuum $\frakc$ can be arbitrarily large. However, $C(K)$ always admits an equivalent strictly convex norm when $K$ is a separable compact space.
	
The aim of this paper is to consistently answer the full version of Question by providing a consistent example of a Banach space $C(K)$ of density $\omega_1<\frakc$ with no strictly convex nor sequentially Kadets--Klee renormings. For a trivial density reason, the constructed space $C(K)$ cannot contain any isomorphic copy of $\elli$, but it has an additional feature of being \textit{Grothendieck}, that is, every weak* convergent sequence in the dual space $C(K)^*$ is weakly convergent (see \cite{GK21,KKLPS}). The main result of the paper thus reads as follows:

\begin{maintheorem*}\label{thm:main}
    It is consistent with \textup{\textsf{ZFC}} that the inequality $\omega_1<\frakc$ holds and there exists a compact space $K$ such that the Banach space $C(K)$ is Grothendieck, has density $\omega_1$, and does not admit any equivalent norm which is strictly convex or sequentially Kadets--Klee. 
\end{maintheorem*}

To establish Main Theorem, we apply the set-theoretic method of forcing. The notion used by us is the side-by-side product $\SS_{\omega_2}$ of $\omega_2$ many copies of the Sacks forcing and the obtained compact space $K$ is the Stone space of the ground model Boolean algebra $\wo/Fin$ (see Corollary \ref{cor:main2} for details). Note that the space $C(St(\wo/Fin))$ in the ground model has all of the properties mentioned in Main Theorem, except for the density which in the ground model is $\frakc$. Consequently, our result is relevant not only for Banach space theoretic investigations, but also presents some new preservation properties of the forcing $\SS_{\omega_2}$: it preserves the lack of strictly convex and of SKK renormings of the Banach space of continuous functions on the Stone space of the ground model Boolean algebra $\wo/Fin$, or, more generally, of ground model Boolean algebras with the Interpolation Property and the Countable Non-Completeness Property (see Theorem \ref{theorem:main_SC}). Note that similar results have already been obtained for the preservation of the Grothendieck and Nikodym properties of ground model $\sigma$-complete Boolean algebras by various classes of forcings (see \cite{Bre06,SZNik,SZForExt,SZAdding}) or, quite related, for the preservation of the lack of non-trivial convergent sequences in ground model extremally disconnected compact spaces by the random and Laver forcings (see \cite{Dow24,DF07,SZAdding}).

\section*{Acknowledgements}

The authors would like to thank Arturo Mart\'inez-Celis for explaining to them in detail why the Diamond Principle $\Diamond(\non(\mM))$ holds in any side-by-side Sacks extension and so why there are Suslin trees in those models (see Remark \ref{remark:suslin_sbs_Sacks}).

\section{Preliminaries\label{sec:prelim}}

By $\omega$ we denote the first infinite (countable) cardinal number, by $\omega_1$ and $\omega_2$ the first and second uncountable cardinal numbers, respectively, and by $\frakc$ the cardinality of the real line $\R$ (i.e. the \textit{continuum}, $\frakc=2^\omega$). 

For a set $X$, by $|X|$ we denote its cardinality, and by $\wp(X)$ and $\finsub{X}$ the collections of all sets of $X$ and of all finite sets of $X$, respectively. We will usually write $Fin=\fso$. For a subset $A\sub X$, $\chi_A\colon X\to\{0,1\}$ denotes the characteristic function of $A$ in $X$. We also say that two sequences $\seq{a_i}{i\in I}$ and $\seq{b_j}{j\in J}$ of elements of $X$ are \textit{disjoint} if $\big\{a_i\colon i\in I\big\}\cap\big\{b_j\colon j\in J\big\}=\emptyset$.

For a Boolean algebra $\aA$, by $\wedge$, $\vee$, $\sm$, $0_\aA$, $1_\aA$, $\le$ we denote its operations, zero and unit elements, and order, respectively. A collection $\seq{A_i}{i\in I}$ of elements of $\aA$ is an \textit{antichain} if $A_i\wedge A_j=0_\aA$ for all $i\neq j\in I$. Two antichains $\seq{A_i}{i\in I}$ and $\seq{B_j}{j\in J}$ in $\aA$ are \textit{disjoint} if $A_i\wedge B_j=0_\aA$ for all $i\in I$, $j\in J$. Recall that $\aA$ is \textit{$\sigma$-complete} if every countable antichain $\seqn{A_n}$ has the supremum $\bigvee_{n\io}A_n$ in $\aA$.

The Stone space of a Boolean algebra $\aA$ is denoted by $St(\aA)$. For each element $A\in\aA$, by $\clopen{A}$ we denote the corresponding clopen subset of $St(\aA)$. Recall that $St(\aA)$ is a totally disconnected compact Hausdorff space and the family $\{\clopen{A}\colon A\in\aA\}$ constitutes a base of the topology of $St(\aA)$. Note also that the Stone spaces $St(\wo)$ and $St(\wo/Fin)$ are usually identified with the \v{C}ech--Stone compactification $\bo$ of the discrete space $\omega$ and its remainder $\omega^*=\bo\sm\omega$, respectively. 

For a topological space $X$, by $w(X)$ and $d(X)$ we denote its weight and density, respectively. If $K$ is a compact (Hausdorff) space, then $C(K)$ denotes the Banach space of all continuous real-valued functions endowed by default with the supremum norm $\|\cdot\|_\infty$. Recall that by the Riesz--Markov--Kakutani representation theorem the dual space $C(K)^*$ is isometrically isomorphic to the Banach space of all signed Radon measures on $K$ endowed with the variation norm. Note also that we have $w(K)=d(C(K))$ and so, for a Boolean algebra $\aA$, $|\aA|=w(St(\aA))=d\big(C(St(\aA))\big)$. For every $f\in C(K)$, we set $\supp(f)=\overline{\{x\in K\colon f(x)\neq 0\}}$. 

The Banach spaces $c_0$, $\elli$, and $\elli(\Gamma)$ have their usual (real-valued) definitions.

Let $\pP$ be a property of norms of Banach spaces, e.g. $\pP$=strictly convex or $\pP$=LUR. For a Banach space $(X,\|\cdot\|)$, we say that $X$ \textit{has a $\pP$ renorming}, if there is an equivalent norm $\trnorm{\,\cdot\,}$ on $X$ which has property $\pP$.

\subsection{The side-by-side Sacks forcing}
 
$V$ always denotes a model of \textsf{ZFC}. Throughout the whole paper, except for Section \ref{sec:adding}, $\kappa$ denotes a regular cardinal number $\ge\omega_2$.

In our forcing arguments we will mostly use the side-by-side Sacks forcing. Our nomenclature follows that of the paper of Baumgartner \cite{Bau85}. By $\SS$ we denote the \textit{Sacks} forcing (or the \textit{perfect-set} forcing), i.e. the collection of all perfect non-empty subtrees of the full tree $2^{<\omega}$ (i.e. $s\in\SS$ if and only if for all $v\in s$ there are incomparable $w,u\in s$ extending $v$), together with the ordering $\le=\sub$. By $\Sacks$ we denote the  \textit{countably supported side-by-side product} of $\kappa$ many Sacks forcings $\SS$, i.e. the collection of all functions $s\colon\dom(s)\to\SS$ with the domain $\dom(s)$ being a countable subset of $\kappa$, endowed with the ordering $\le$ defined as follows: for $s,s'\in\Sacks$, $s'\le s$ if $\dom(s)\sub\dom(s')$ and $s'(\alpha)\le s(\alpha)$ for every $\alpha\in\dom(s)$. 

If $s\in\mathbb{S}$ and $v\in s$, then $s|v=\big\{w\in s\colon\ w\sub v\text{ or }v\sub w\big\}\in\SS$. If $n\io$, then $\lev(n,s)$ denotes the \textit{$n$-th forking level} of $s$, i.e. $v\in\lev(n,s)$ if and only if there are exactly $n$ vertices $w\subsetneq v$ such that $s$ splits at $w$, and no $w\subsetneq v$ has this property as well. 
Note that $|\lev(n,s)|=2^n$.

Let $s,s'\in\Sacks,F\in\finsub{\dom(s)}$, and $n\io$. We put $\lev(F,n,s)=\big\{\sigma\colon\ \dom(\sigma)=F\ \&\ \forall\alpha\in F\colon\ \sigma(\alpha)\in \lev(n,s(\alpha))\big\}$. Note that $|\lev(F,n,s)|=2^{n|F|}$. We write $s'\le_{F,n}s$, if $s'\le s$ and $\lev(F,n,s')=\lev(F,n,s)$. If $\sigma\colon F\to 2^{<\omega}$ is such that $\sigma(\alpha)\in s(\alpha)$ for every $\alpha\in F$, then we write $s|\sigma$ for the condition defined as $(s|\sigma)(\alpha)=s(\alpha)$ for $\alpha\in\dom(s)\sm F$ and $(s|\sigma)(\alpha)=s(\alpha)|\sigma(\alpha)$ for $\alpha\in F$.

If $G$ is an $\Sacks$-generic filter over $V$ and $\tau\in V$ is an $\Sacks$-name, then by $(\tau)_G$ we denote the interpretation of $\tau$ in $V[G]$. We will also use the convention that if $x$ is an element of the ground model $V$ (resp. of the extension $V[G]$), then $\check{x}$ (resp. $\dot{x}$) is an $\Sacks$-name such that $(\check{x})_G=x$ (resp. $(\dot{x})_G=x$).

The following lemma is standard, cf. \cite[Lemmas 1.5 and 1.7]{Bau85}.

\begin{lemma}\label{lemma:baumgartner}
Let $s\in\Sacks$, $N\io$, and $F\in\finsub{\dom(s)}$.
\begin{enumerate}[(a),itemsep=1mm]
	\item\label{lemma:baumgartner:antichain} $\big\{s|\sigma\colon\ \sigma\in \lev(F,N,s)\big\}$ is an antichain in $\Sacks$ and $s=\bigcup_{\sigma\in \lev(F,N,s)}s|\sigma$.
	\item\label{lemma:baumgartner:dense} If $\sigma\in \lev(F,N,s)$ and $p\le s|\sigma$, then there exists $q\le_{F,N}s$ such that $q|\sigma=p$.
\end{enumerate}\noproof
\end{lemma}

\section{Technical lemmas}

In this section we will introduce some auxiliary notions and facts concerning Boolean algebras, norms on $C(K)$-spaces, and the side-by-side Sacks forcing.

\subsection{Auxiliary notions for Boolean algebras\label{sec:aux_ba}}

Fix an infinite Boolean algebra $\aA$ and an antichain $\seqn{C_n}$ in $\aA$ such that $C_n\neq 0_\aA$ for all $n\io$. For $A\iA$ set
\[\alpha(A)=\big|\big\{n\io\colon A\wedge C_n=0_\aA\big\}\big|.\]
Put
\[\xX=\big\{(A,B)\colon\ A,B\iA,\ A\wedge B=0_\aA,\ \alpha(A\vee B)=\infty\big\}.\]
Let $o=\big(0_\aA,0_\aA)$ and note that $o\in\xX$. For $x=(A,B)\in\xX$ let
\[l(x)=A,\quad r(x)=B,\quad\text{and}\quad\supp(x)=A\vee B=l(x)\vee r(x).\]
For $x,y\in\xX$ set
\[x\oplus y=\big(l(x)\vee l(y), r(x)\vee r(y)\big).\]
For every $x\in\xX$ let also
\[\dD(x)=\big\{Z\iA\colon Z\wedge\supp(x)=0_\aA\big\}.\]
Moreover, for all $x,u\in\xX$ we write $u\preceq x$ if $l(u)\le l(x)$ and $r(u)\le r(x)$.
Finally, if $x\in\xX$ and $Z\in\dD(x)$, then set
\[\fF\big(x,Z\big)=\big\{z\in\xX\colon\ x\preceq z,\ Z\in\dD(z)\big\};\]
note that $x\in\fF(x,Z)$.

\subsection{Auxiliary notions for norms\label{sec:aux_norms}}

Fix the objects as in Section \ref{sec:aux_ba}. For each $x\in\xX$ set 
\[\varphi(x)=\chi_{\clopen{l(x)}}-\chi_{\clopen{r(x)}},\]
and note that $\varphi(x)\in C(St(\aA))$ and $\clopen{\supp(x)}=\supp(\varphi(x))$. Also, $\varphi(o)=0$.

Let $\trnorm{\,\cdot\,}$ be an equivalent norm on $C(St(\aA))$, that is, $a\trnorm{\,\cdot\,}\le\|\cdot\|_\infty\le b\trnorm{\,\cdot\,}$ for some $a,b>0$. For every $u\preceq x\in\xX$ and $Z\in\dD(x)$ define:
\[m(u,x,Z\big)=\inf_{z\in\fF(x,Z)}\trnorm{\varphi(u)-\varphi(z)},\]
\[M(u,x,Z\big)=\sup_{z\in\fF(x,Z)}\trnorm{\varphi(u)-\varphi(z)},\]
\[d(u,x,Z\big)=M(u,x,Z)-m(u,x,Z).\]
We will say that the triple $(m,M,d)$ is \textit{associated} with the norm $\trnorm{\,\cdot\,}$. 

\medskip

We will frequently use the following standard lemma (cf. \cite{Ale01,AB88}); we provide its (partial) proof for the convenience of the reader.

\begin{lemma}\label{lemma:basic}
    Let $x\in\xX$, $u\preceq x$, and $Z\in\dD(x)$ be such that $\alpha(Z\vee\supp(x))=\infty$. 
    \begin{enumerate}[(i),itemsep=1mm]
        \item\label{lemma:basic:bigger_Z} If $Z'\in\dD(x)$ and $Z\le Z'$, then $d(u,x,Z')\le d(u,x,Z)$.
        \item\label{lemma:basic:bigger_x} If $x'\in\fF(x,Z)$, then $d(u,x',Z)\le d(u,x,Z)$.
        \item\label{lemma:basic:exists} For any finite set $\yY\sub\xX$ such that $y\preceq x$ for all $y\in\yY$, there is $z\in\fF(x,Z)\sm\{x\}$ such that  $d(y,z,Z)\le d(y,x,Z)/2$ for all $y\in\yY$.
        \item\label{lemma:basic:addition} If $z\in\xX$ is such that $\supp(z)\wedge\supp(x)=0_\aA$ and $Z\in\dD(x\oplus z)$, then $d(u\oplus z,x\oplus z,Z)\le d(u,x,Z)$.
    \end{enumerate}
\end{lemma}
\begin{proof}
    Items (i), (ii), and (iv) are immediate, so we only prove item (iii). For this we will first show that whenever $u \preceq v \in \xX$ and $Z \in \dD(v)$, we have
    \[\tag{$*$}\trnorm{\varphi(v) - \varphi(u)} \le \frac{M(u,v,Z) + m(u, v, Z)}{2}.\]
    To this end, let $\eps > 0$ and find $w \in \fF(v, Z)$ such that
    \[\trnorm{\varphi(w) - \varphi(u)} < m(u,v,Z) + \eps.\]
    Notice that
    \[2 \varphi(v) - \varphi(w) = \varphi\Bigl(\bigl(l(v) \vee (r(w) \sm r(v))\, ,\, r(v) \vee (l(w) \sm l(v)) \bigr)\Bigr).\]
    In particular, the argument of the right-hand side is an element of $\fF(v, Z)$ and thus
    \[\trnorm{\bigl(2 \varphi(v) - \varphi(w)\bigr) - \varphi(u)} \le M(u,v,Z).\]
    We now have
    \[\trnorm{\varphi(v) - \varphi(u)} \le \frac{\trnorm{\bigl(2 \varphi(v) - \varphi(w)\bigr) - \varphi(u)} + \trnorm{\varphi(w) - \varphi(u)}}{2}<\]
    \[< \frac{M(u,v,Z) + m(u,v,Z) + \eps}{2},\]
    which proves ($*$), as $\eps>0$ was arbitrary.
    
    Returning to the proof of (iii), consider first the case where $\yY$ is a singleton, i.e. $\yY = \{y\}$ for some $y\in\xX$ such that $y\preceq x$. If $d(y,x,Z)=0$, then pick any $n\io$ such that $\big(C_n\vee\supp(x)\big)\wedge Z=0_\aA$ and let $z=x\oplus\big(C_n,0_\aA\big)$. If $d(y,x,Z)\neq 0$, then we proceed as follows. Note that
    \[\trnorm{\varphi(x)-\varphi(y)}\le M(y,x,Z),\]
    so we can find $z \in \fF(x,Z)$ such that
    \[\trnorm{\varphi(z) - \varphi(y)} \ge \frac{3 M(y,x,Z) + \trnorm{\varphi(x) - \varphi(y)}}{4}.\]
    Combining ($*$), the latter inequality, and (ii), we obtain
    \[m(y,z,Z) \ge 2 \trnorm{\varphi(z) - \varphi(y)} - M(y,z,Z)\ge\]
    \[\ge \frac{3 M(y,x,Z) + \trnorm{\varphi(x) - \varphi(y)}}{2} - M(y,z,Z)\ge\]
    \[\ge \frac{3 M(y,x,Z) + \trnorm{\varphi(x) - \varphi(y)}}{2} - M(y,x,Z)=\]
    \[= \frac{M(y,x,Z) + \trnorm{\varphi(x) - \varphi(y)}}{2}.\]
    We can now estimate:
    \[d(y,z,Z) = M(y,z,Z) - m(y,z,Z)\le\]
    \[\le M(y,x,Z) - \frac{M(y,x,Z) + \trnorm{\varphi(x) - \varphi(y)}}{2}\le\]
    \[= \frac{M(y,x,Z) - \trnorm{\varphi(x) - \varphi(y)}}{2}\le \frac{M(y,x,Z) - m(y,x,Z)}{2}=\]
    \[=d(y,x,Z)/2,\]
    which yields the required inequality. Observe that $d(y,x,Z)\neq0$ implies that $z\neq x$.
    
    For the general case, we proceed by induction on the size of $\yY$. If $\yY' = \yY \cup \{y'\}$ and there is $z \in \fF(x,Z)$ that satisfies the conclusion for $\yY$, we can use the above argument to find $z' \in \fF(z,Z)$ such that $d(y',z',Z)\le d(y',z,Z)/2$. Then, by \ref{lemma:basic:bigger_x}, $z'$ is the desired element for $\yY'$.
\end{proof}

\subsection{A forcing lemma for renormings\label{sec:aux_lur}}

In $V$, fix the objects as in Section \ref{sec:aux_ba}. Let $\dot{m},\dot{M},\dot{d}$ be $\Sacks$-names and assume that a condition $s\in\Sacks$ forces that the triple $(\dot{m},\dot{M},\dot{d})$ is associated with some equivalent norm $\trnorm{\,\cdot\,}$ on $C(St(\check{\aA}))$ in $V^{\Sacks}$.

\begin{lemma}\label{lemma:forcing_renorming}
    Let $x_0=\big(A_0,B_0\big)\in\xX$, $Z_0\in\dD(x_0)$ with $\alpha\big(Z_0\vee\supp(x_0)\big)=\infty$, $L,M,N\io$, and $F\in\finsub{\dom(s)}$. Set $\gamma=|\lev(F,N,s)|$. 
    
    Then, there are a condition $s'\in\Sacks$ with $s'\le_{F,N} s$, disjoint antichains $A_1,\ldots,A_\gamma$ and $B_1,\ldots,B_\gamma$ in $\dD(x_0)$, and a sequence $\Big\langle\big\langle i_1^m,\ldots,i_\gamma^m\big\rangle\colon 1\le m\le L\Big\rangle$ of pairwise disjoint strictly increasing sequences of natural numbers, all strictly larger than $M$, such that:
    \begin{itemize}[itemsep=1mm]
        \item for each $1\le k\le\gamma$, setting
        \[\quad\quad x_k=x_0\oplus\bigoplus_{1\le i\le k}\big(A_i,B_i\big)\]
        and $Z_k=Z_{k-1}\vee\bigvee_{1\le m\le L}C_{i_k^m}$, we have $x_k\in\fF\big(x_{k-1},Z_{k-1}\big)\sm\{x_{k-1}\}$ and $Z_k \in \dD(x_k)$,
        \item $s'\forces \exists 1\le k\le \check{\gamma}\ \forall 0\le l<k\colon\ \dot{d}\big(\check{o},\check{x}_k,\check{Z}_k\big)\le\dot{d}\big(\check{o},\check{x}_l,\check{Z}_l\big)/2$.
    \end{itemize}
\end{lemma}
\begin{proof}
    Let $\seq{\sigma_k}{1\le k\le\gamma}$ be an enumeration of $\lev(F,N,s)$. The proof will proceed by induction on $k=1,\ldots,\gamma$; we will however only present the first two steps (for $k=1,2$), as the further ones should be immediately clear.
    
    Let $k=1$. Set $s_0=s$. By Lemma \ref{lemma:basic}\ref{lemma:basic:exists}, there exist $x_1=\big(A,B\big)\in\fF\big(x_0,Z_0\big)\sm\{x_0\}$ and a forcing condition $p_1\le s_0|\sigma_1$ such that 
    \[\tag{1}p_1\forces \dot{d}\big(\check{o},\check{x}_1,\check{Z}_0\big)\le\dot{d}\big(\check{o},\check{x}_0,\check{Z}_0\big)/2.\]
    Since $(A,B)\in\xX$ and so $\alpha(A\vee B)=\infty$, there are distinct numbers $i_1^1,\ldots,i_1^L>M$ such that
    \[(A\vee B)\wedge\bigvee_{1\le m\le L}C_{i_1^m}=0_\aA.\]
    Set
    \[Z_1=Z_0\vee\bigvee_{1\le m\le L}C_{i_1^m},\]
    so $Z_0\le Z_1$ and $Z_1\in\dD(x_1)$. According to Lemma \ref{lemma:basic}\ref{lemma:basic:bigger_Z}, we have
    \[p_1\forces \dot{d}\big(\check{o},\check{x}_1,\check{Z}_1\big)\le\dot{d}\big(\check{o},\check{x}_1,\check{Z}_0\big),\]
    so, by (1), it holds
    \[\tag{2}p_1\forces \dot{d}\big(\check{o},\check{x}_1,\check{Z}_1\big)\le\dot{d}\big(\check{o},\check{x}_0,\check{Z}_0\big)/2.\]
    Set $A_1=A\sm A_0$ and $B_1=B\sm B_0$. Then,
    \[\big(A_1\vee B_1\big)\wedge Z_1=0_\aA\quad\text{and}\quad x_1=x_0\oplus\big(A_1,B_1\big).\]
    
    By Lemma \ref{lemma:baumgartner}\ref{lemma:baumgartner:dense}, there is a condition $s_1\le_{F,N} s_0$ such that $s_1|\sigma_1=p_1$, so that by (2) we have
    \[\tag{3}s_1|\sigma_1\forces \dot{d}\big(\check{o},\check{x}_1,\check{Z}_1\big)\le\dot{d}\big(\check{o},\check{x}_0,\check{Z}_0\big)/2.\]
    Note however that \textit{a priori}, by Lemma \ref{lemma:basic}\ref{lemma:basic:bigger_Z}--\ref{lemma:basic:bigger_x}, for any $1<l\le\gamma$ we only have
    \[s_1|\sigma_l\forces \dot{d}\big(\check{o},\check{x}_1,\check{Z}_1\big)\le\dot{d}\big(\check{o},\check{x}_0,\check{Z}_0\big).\]

\medskip

    Let us now proceed with the step for $k=2$; it is almost the same as for $k=1$. By Lemma \ref{lemma:basic}\ref{lemma:basic:exists}, there exist $x_2=\big(A',B'\big)\in\fF\big(x_1,Z_1\big)\sm\{x_1\}$ and a forcing condition $p_2\le s_1|\sigma_2$ such that 
    \[\tag{1'}p_2\forces \dot{d}\big(\check{o},\check{x}_2,\check{Z}_1\big)\le\dot{d}\big(\check{o},\check{x}_1,\check{Z}_1\big)/2.\]
    Since $(A',B')\in\xX$ and so $\alpha(A'\vee B')=\infty$, there are distinct numbers $i_2^1,\ldots,i_2^L>\max\big\{i_1^1,\ldots,i_1^L\big\}$ such that
    \[(A'\vee B')\wedge\bigvee_{1\le m\le L}C_{i_2^m}=0_\aA.\]
    Set
    \[Z_2=Z_1\vee\bigvee_{1\le m\le L}C_{i_2^m},\]
    so $Z_1\le Z_2$ and $Z_2\in\dD(x_2)$. According to Lemma \ref{lemma:basic}\ref{lemma:basic:bigger_Z}, we have
    \[p_2\forces \dot{d}\big(\check{o},\check{x}_2,\check{Z}_2\big)\le\dot{d}\big(\check{o},\check{x}_2,\check{Z}_1\big),\]
    so, by (1'), it holds
    \[\tag{2'}p_2\forces \dot{d}\big(\check{o},\check{x}_2,\check{Z}_2\big)\le\dot{d}\big(\check{o},\check{x}_1,\check{Z}_1\big)/2.\]
    Set $A_2=A'\sm\big(A_0\vee A_1\big)$ and $B_2=B'\sm\big(B_0\vee B_1\big)$. Then,
    \[\big(A_2\vee B_2\big)\wedge Z_2=0_\aA\quad\text{and}\quad x_2=x_0\oplus\big(A_1,B_1\big)\oplus\big(A_2,B_2\big).\]

\medskip

     By Lemma \ref{lemma:baumgartner}\ref{lemma:baumgartner:dense}, there is a condition $s_2\le_{F,N} s_1$ such that $s_2|\sigma_2=p_2$, so that by (2') we have
    \[\tag{3'}s_2|\sigma_2\forces \dot{d}\big(\check{o},\check{x}_2,\check{Z}_2\big)\le\dot{d}\big(\check{o},\check{x}_1,\check{Z}_1\big)/2.\]
    Note however that \textit{a priori}, by Lemma \ref{lemma:basic}\ref{lemma:basic:bigger_Z}--\ref{lemma:basic:bigger_x}, for any $2<l\le\gamma$ and $0\le m\le 1$ we only have
    \[s_2|\sigma_l\forces \dot{d}\big(\check{o},\check{x}_2,\check{Z}_2\big)\le\dot{d}\big(\check{o},\check{x}_m,\check{Z}_m\big).\]
    On the other hand, it also holds, by (3') and again Lemma \ref{lemma:basic}\ref{lemma:basic:bigger_Z}--\ref{lemma:basic:bigger_x}, that 
    \[\tag{4'}s_2|\sigma_2\forces \dot{d}\big(\check{o},\check{x}_2,\check{Z}_2\big)\le\dot{d}\big(\check{o},\check{x}_0,\check{Z}_0\big)/2.\]

\medskip

    As said at the beginning of the proof, we proceed in a similar way also in the $k$-th steps for $k=3,\ldots,\gamma$. After we have finished, we set $s'=s_\gamma$ and appeal to (3), (3'), (4'), etc., and Lemma \ref{lemma:baumgartner}\ref{lemma:baumgartner:antichain} to get that
    \[s'\forces \exists 1\le k\le \check{\gamma}\ \forall 0\le l<k\colon\ \dot{d}\big(\check{o},\check{x}_k,\check{Z}_k\big)\le\dot{d}\big(\check{o},\check{x}_l,\check{Z}_l\big)/2.\]
\end{proof}

\section{Main results \label{sec:main}}

Let us start with the following definition introduced by Seever \cite{See68}.

\begin{definition}\label{def:property_i}
    A Boolean algebra $\aA$ has the \textit{Interpolation Property} (in short, \textit{(I)}) if for any two disjoint countable antichains $\seqn{A_n}$ and $\seqn{B_n}$ in $\aA$ there is $A\iA$ such that $A_n\le A$ and $A\wedge B_n=0_\aA$ for all $n\io$.
\end{definition}

It is immediate that every $\sigma$-complete Boolean algebra has the Interpolation Property. By \cite[Theorem A]{See68}, a Boolean algebra $\aA$ has the property if and only if the Stone space $St(\aA)$ is an \textit{F-space}, that is, disjoint open $\mathbb{F}_\sigma$-sets have disjoint closures in $St(\aA)$. Recall that $\omega^*$ is a typical example of a non-extremally disconnected totally disconnected compact F-space, and so the quotient Boolean algebra $\wo/Fin$ has property (I).

We are in the position to prove the first two main results of this section.

\begin{theorem}\label{theorem:main_wMLUR}
Let $\aA\in V$ be a Boolean algebra which has property (I) in $V$. Let $G$ be an $\Sacks$-generic filter over $V$. Then, in $V[G]$ the Banach space $C(St(\aA))$ admits no equivalent wMLUR norm.
\end{theorem}
\begin{proof}
Starting in the ground model $V$, let the objects as in Section \ref{sec:aux_ba} be given for the algebra $\aA$. 

In the extension $V[G]$, let $\trnorm{\,\cdot\,}$ be an equivalent norm on the Banach space $C(St(\aA))$ and let the functions $m,M,d$ be defined for $\trnorm{\,\cdot\,}$ as in Section \ref{sec:aux_norms}. For the sake of contradiction assume that $\trnorm{\,\cdot\,}$ is wMLUR.

Let us go back to $V$. Let $\dot{n},\dot{m},\dot{M},\dot{d}$ be $\Sacks$-names such that $(\dot{n})_G=\trnorm{\,\cdot\,}$, $(\dot{m})_G=m$, $(\dot{M})_G=M$, and $(\dot{d})_G=d$, and let $s\in G$ be a condition forcing that $\dot{n}$ is an equivalent wMLUR norm on $C(St(\check{\aA}))$ and that the triple $(\dot{m},\dot{M},\dot{d})$ is associated with $\dot{n}$.

\medskip

We will find a condition $s^*\in\Sacks$, $s^*\le s$, forcing that $\dot{n}$ is not a wMLUR norm, getting a contradiction. 
First, by the inductive use of Lemma \ref{lemma:forcing_renorming} (with $L=2$) we obtain:
\begin{itemize}[itemsep=1mm]
    \item a sequence $\seqN{s_N}$ of conditions in $\Sacks$ such that $s_0=s$ and $s_{N+1}\le_{F_N,N}s_N$ for every $N\io$, where we enumerate $\dom(s_N)=\big\{\alpha_k^N\colon k\io\big\}$ and we put $F_N=\big\{\alpha_k^i\colon i,k\le N\big\}$ and $\gamma_N=\big|\lev\big(F_N,N,s_N\big)\big|$,
    \item sequences $\Big\langle\big\langle A_{N,1},\ \ldots,\ A_{N,\gamma_N}\big\rangle\colon\ N\io\Big\rangle$ and $\Big\langle\big\langle B_{N,1},\ \ldots,\ B_{N,\gamma_N}\big\rangle\colon\ N\io\Big\rangle$ of finite pairwise disjoint antichains in $\aA$ such that
    \[\quad\quad\bigvee_{i=1}^{\gamma_N} A_{N,i}\wedge\bigvee_{j=1}^{\gamma_{N'}}B_{N',j}=0_\aA\]
    for every $N, N'\io$,
    \item and sequences $\Big\langle\big\langle i_{N,1},\ \ldots,\ i_{N,\gamma_N}\big\rangle\colon\ N\io\Big\rangle$ and $\Big\langle\big\langle j_{N,1},\ \ldots,\ j_{N,\gamma_N}\big\rangle\colon\ N\io\Big\rangle$ of finite strictly increasing sequences of natural numbers such that $i_{N,k}\neq j_{N',l}$, $i_{N,\gamma_N}<i_{N+1,1}$, and $j_{N',\gamma_{N'}}<j_{N'+1,1}$ for all $N,N'\io$, $1\le k\le\gamma_N$, and $1\le l\le\gamma_{N'}$,
\end{itemize}
having also the following properties:
\begin{itemize}[itemsep=1mm]\renewcommand{\labelitemi}{\tiny{$\blacksquare$}}
    \item for all $N,N'\io$ we have
    \[\quad\quad\bigvee_{i=1}^{\gamma_N}\big(A_{N,i}\vee B_{N,i}\big)\wedge\bigvee_{k=1}^{\gamma_{N'}}\big(C_{i_{N',k}}\vee C_{j_{N',k}}\big)=0_\aA,\]
    \item for each $N\io$ and $1\le k\le\gamma_N$, setting
    \[\quad\quad x_{N,k}=\bigg(\bigoplus_{N'<N}\bigoplus_{1\le i\le\gamma_{N'}}\big(A_{N',i},B_{N',i}\big)\bigg)\oplus\bigoplus_{1\le i\le k}\big(A_{N,i},B_{N,i}\big)\]
    and
    \[\quad\quad Z_{N,k}=\bigg(\bigvee_{N'<N}\bigvee_{1\le l\le\gamma_{N'}}\big(C_{i_{N',l}}\vee C_{j_{N',l}}\big)\bigg)\vee\bigvee_{1\le l\le k}\big(C_{i_{N,l}}\vee C_{j_{N,l}}\big),\]
    as well as auxiliarily
    \[\quad\quad x_{0,0}=o\quad\text{and}\quad Z_{0,0}=0_\aA,\]
    \[\quad\quad x_{N,0}=x_{N-1,\gamma_{N-1}}\quad\text{and}\quad Z_{N,0}=Z_{N-1,\gamma_{N-1}}\quad\text{if }N>0,\]
    we have
    \[\quad\quad x_{N,k}\in\fF\big(x_{N,k-1}\ ,\ Z_{N,k-1}\big)\sm\big\{x_{N,k-1}\big\},\]
    \item for each $N\io$ we have
    \[\tag{$1$}\quad\quad s_N\forces\exists 1\le k\le\check{\gamma}_{N}\ \forall 0\le l<k\colon\ \dot{d}\big(\check{o},\check{x}_{N,k},\check{Z}_{N,k}\big)\le\dot{d}\big(\check{o},\check{x}_{N,l},\check{Z}_{N,l}\big)/2.\]
\end{itemize}
Then, let $s^*\in\Sacks$ be a fusion condition of the fusion sequence $\seqN{s_N}$, that is, such condition $s^*$ that $s^*\le_{F_N,N} s_N$ holds for every $N\io$ (see \cite[Lemma 1.6]{Bau85}). From (1) it follows that
\begin{align*}
    \tag{$2$}s^*\forces\forall N\in\check{\omega}\ \exists 1\le k_N\le\check{\gamma}_N\ &\forall 0\le l<k_N\colon\ \\&\dot{d}\big(\check{o},\check{x}_{N,k_N},\check{Z}_{N,k_N}\big)\le\dot{d}\big(\check{o},\check{x}_{N,l},\check{Z}_{N,l}\big)/2.
\end{align*}

\medskip

As the Boolean algebra $\aA$ has property (I) in $V$, there are disjoint elements $A,B,C\in\aA$ such that for all $N\io$ we have
\[\bigvee_{1\le k\le\gamma_N}A_{N,k}\le A\quad,\quad\bigvee_{1\le k\le\gamma_N}B_{N,k}\le B\quad,\quad\bigvee_{1\le k\le\gamma_N}C_{i_{N,k}}\le C\quad,\]
and
\[A\wedge\bigvee_{1\le k\le\gamma_N}\big(B_{N,k}\vee C_{i_{N,k}}\vee C_{j_{N,k}}\big)=0_\aA,\]
\[B\wedge\bigvee_{1\le k\le\gamma_N}\big(A_{N,k}\vee C_{i_{N,k}}\vee C_{j_{N,k}}\big)=0_\aA,\]
\[C\wedge\bigvee_{1\le k\le\gamma_N}\big(A_{N,k}\vee B_{N,k}\vee C_{j_{N,k}}\big)=0_\aA.\]
As $(A\vee B\vee C) \wedge C_{j_{N,k}} = 0_\aA$ for all $N \io$ and $1 \le k \le \gamma_N$, we have $\alpha(A\vee B\vee C)=\infty$. Let $x=(A,B)$, so that $x\in\xX\sm\{o\}$ (note that even $\big(A_{0,1},B_{0,1}\big)\neq o$). 
For every $N\io$ and $1\le k\le\gamma_N$ set also
\[D_{N,k}=C\sm\bigg(\bigvee_{N'<N}\bigvee_{1\le l\le \gamma_{N'}}C_{i_{N',l}}\vee\bigvee_{1\le l\le k}C_{i_{N,l}}\bigg),\]
\[y_{N,k}^+=\big(D_{N,k},0_\aA\big)\quad\text{and}\quad y_{N,k}^-=\big(0_\aA,D_{N,k}\big);\]
note that $y_{N,k}^\pm\in\xX$ and
\[\tag{3}x,\, x\oplus y_{N,k}^\pm\in\fF\big(x_{N,k},Z_{N,k}\big).\]

\medskip

Let now $G'$ be an $\Sacks$-generic filter over $V$ containing $s^*$ (and so $s$). Let us go to the extension $V[G']$. Let us also write $\trnorm{\,\cdot\,}=(\dot{n})_{G'}$, $m=(\dot{m})_{G'}$, $M=(\dot{M})_{G'}$, and $d=(\dot{d})_{G'}$. Note that by the assumption the norm $\trnorm{\,\cdot\,}$ is wMLUR in $V[G']$. By (2), for each $N\io$ there is $1\le k_N\le\gamma_N$ such that for every $0\le l<k_N$ we have
\[d\big(o,\, x_{N,k_N},\, Z_{N,k_N}\big)\le d\big(o,\, x_{N,l},\, Z_{N,l}\big)/2.\]
In particular, by Lemma \ref{lemma:basic}\ref{lemma:basic:bigger_Z}--\ref{lemma:basic:bigger_x}, for every $N\io$ it holds
\[d\big(o,\, x_{N+1,k_{N+1}},\, Z_{N+1,k_{N+1}}\big)\le d\big(o,\, x_{N+1,0},\, Z_{N+1,0}\big)/2=\]
\[=d\big(o,\, x_{N,\gamma_{N}},\, Z_{N,\gamma_{N}}\big)/2\le d\big(o,\, x_{N,k_{N}},\, Z_{N,k_{N}}\big)/2.\]
Since $d\big(o,\, x_{0,k_0},\, Z_{0,k_0}\big)<\infty$ (as the norms $\|\cdot\|_\infty$ and $\trnorm{\,\cdot\,}$ are equivalent), it follows that
\[\tag{4}\lim_{N\to\infty}d\big(o,\, x_{N,k_N},\, Z_{N,k_N}\big)=0.\]

For each $N\io$ let $g_N=\chi_{\clopen{D_{N,k_N}}}=\varphi\Big(y_{N,k_N}^+\Big)$, so that $g_N\neq0$ and we have
\[\varphi\bigg(x\oplus y_{N,k_N}^\pm\bigg)=\varphi(x)+\varphi\Big(y_{N,k_N}^\pm\Big)=\varphi(x)\pm g_N.\]
As $\varphi(o)=0$, it follows from (3) that for all $N\io$ it holds
\[m\big(o,\, x_{N,k_N},\, Z_{N,k_N}\big)\le\trnorm{\varphi(x)}\le M\big(o,\, x_{N,k_N},\, Z_{N,k_N}\big)\]
and
\[m\big(o,\, x_{N,k_N},\, Z_{N,k_N}\big)\le\trnorm{\varphi(x)\pm g_N}\le M\big(o,\, x_{N,k_N},\, Z_{N,k_N}\big).\]
Hence, (4) implies that
\[\lim_{N\to\infty}m\big(o,\, x_{N,k_N},\, Z_{N,k_N}\big)=\lim_{N\to\infty}M\big(o,\, x_{N,k_N},\, Z_{N,k_N}\big)=\]
\[=\lim_{N\to\infty}\trnorm{\varphi(x)\pm g_N}=\trnorm{\varphi(x)}.\]
Note that $\trnorm{\varphi(x)}>0$, since $\varphi(x)\neq0$. Consequently, setting
\[f=\varphi(x)/\trnorm{\varphi(x)}\quad\text{and}\quad f_N=g_N/\trnorm{\varphi(x)}\quad\text{for }N\io,\]
we have $\trnorm{f}=1$ and $\lim_{N\to\infty}\trnorm{f\pm f_N}=1$.

Since the clopen set $\clopen{C}$ is compact and infinite, there is a point 
\[t\in\bigcap_{N\io}\clopen{D_{N,k_N}}.\]
Then, for each $N\io$ and the one-point Radon measure $\delta_t$ on $St(\aA)$ concentrated at $t$, we have 
\[\int_{St(\aA)}f_Nd\delta_t=f_N(t)=1/\trnorm{\varphi(x)},\]
so the sequence $\seqN{f_N}$ does not converge weakly to $0$. Consequently, the norm $\trnorm{\,\cdot\,}$ is not wMLUR (in $V[G']$), which is a contradiction. 

It follows that $C(St(\aA))$ does not admit any wMLUR renorming in $V[G]$.
\end{proof}

\begin{theorem}\label{theorem:main_SKK}
Let $\aA\in V$ be a Boolean algebra which has property (I) in $V$. Let $G$ be an $\Sacks$-generic filter over $V$. Then, in $V[G]$ the Banach space $C(St(\aA))$ admits no equivalent SKK norm.
\end{theorem}
\begin{proof}
The proof is similar to the one of Theorem \ref{theorem:main_wMLUR}. So, starting in the ground model $V$, let the objects as in Section \ref{sec:aux_ba} for $\aA$ be given. Then, in $V[G]$, again for the sake of contradiction, assume that an equivalent norm $\trnorm{\,\cdot\,}$ on $C(St(\aA))$ has the SKK property. Again back in $V$, let $\dot{n},\dot{m},\dot{M},\dot{d}$ be $\Sacks$-names such that $(\dot{n})_G=\trnorm{\,\cdot\,}$, $(\dot{m})_G=m$, $(\dot{M})_G=M$, and $(\dot{d})_G=d$, and let $s\in G$ be a condition forcing that $\dot{n}$ is an equivalent SKK norm on $C(St(\check{\aA}))$ and that $(\dot{m},\dot{M},\dot{d})$ is the triple associated with $\dot{n}$. As in the proof of Theorem \ref{theorem:main_wMLUR}, with the aid of Lemma \ref{lemma:forcing_renorming} (with $L=1$) and by the fusion argument, we construct:
\begin{itemize}[itemsep=1mm]
    \item a condition $s^{**}\in\Sacks$, $s^{**}\le s$,
    \item sequences $\Big\langle\big\langle A_{N,1},\ \ldots,\ A_{N,\gamma_N}\big\rangle\colon\ N\io\Big\rangle$ and $\Big\langle\big\langle B_{N,1},\ \ldots,\ B_{N,\gamma_N}\big\rangle\colon\ N\io\Big\rangle$ of finite pairwise disjoint antichains in $\aA$ such that
    \[\quad\quad\bigvee_{i=1}^{\gamma_N} A_{N,i}\wedge\bigvee_{j=1}^{\gamma_{N'}}B_{N',j}=0_\aA\]
    for every $N, N'\io$,
    \item and a sequence $\Big\langle\big\langle i_{N,1},\ \ldots,\ i_{N,\gamma_N}\big\rangle\colon\ N\io\Big\rangle$ of finite strictly increasing sequences of natural numbers such that $i_{N,\gamma_N}<i_{N+1,1}$ for all $N\io$,
\end{itemize}
having also the following properties:
\begin{itemize}[itemsep=1mm]\renewcommand{\labelitemi}{\tiny{$\blacksquare$}}
    \item for all $N,N'\io$ we have
    \[\quad\quad\bigvee_{i=1}^{\gamma_N}\big(A_{N,i}\vee B_{N,i}\big)\wedge\bigvee_{k=1}^{\gamma_{N'}}C_{i_{N',k}}=0_\aA,\]
    \item for each $N\io$ and $1\le k\le\gamma_N$, setting
    \[\quad\quad x_{N,k}=\bigg(\bigoplus_{N'<N}\bigoplus_{1\le i\le\gamma_{N'}}\big(A_{N',i},B_{N',i}\big)\bigg)\oplus\bigoplus_{1\le i\le k}\big(A_{N,i},B_{N,i}\big)\]
    and
    \[\quad\quad Z_{N,k}=\bigg(\bigvee_{N'<N}\bigvee_{1\le l\le\gamma_{N'}}C_{i_{N',l}}\bigg)\vee\bigvee_{1\le l\le k}C_{i_{N,l}},\]
    as well as auxiliarily
    \[\quad\quad x_{0,0}=o\quad\text{and}\quad Z_{0,0}=0_\aA,\]
    \[\quad\quad x_{N,0}=x_{N-1,\gamma_{N-1}}\quad\text{and}\quad Z_{N,0}=Z_{N-1,\gamma_{N-1}}\quad\text{if }N>0,\]
    we have
    \[\quad\quad x_{N,k}\in\fF\big(x_{N,k-1}\ ,\ Z_{N,k-1}\big)\sm\big\{x_{N,k-1}\big\},\]
    \item it holds
    \begin{align*}
    \tag{$1$}s^{**}\forces\forall N\in\check{\omega}\ \exists 1\le k_N\le\check{\gamma}_N\ &\forall 0\le l<k_N\colon\ \\&\dot{d}\big(\check{o},\check{x}_{N,k_N},\check{Z}_{N,k_N}\big)\le\dot{d}\big(\check{o},\check{x}_{N,l},\check{Z}_{N,l}\big)/2.
\end{align*}
\end{itemize}

\medskip

As the Boolean algebra $\aA$ has property (I) in $V$, there are disjoint elements $A,B\in\aA$ such that for all $N\io$ we have
\[\bigvee_{1\le k\le\gamma_N}A_{N,k}\le A\quad,\quad\bigvee_{1\le k\le\gamma_N}B_{N,k}\le B\quad,\]
and
\[A\wedge\bigvee_{1\le k\le\gamma_N}\big(B_{N,k}\vee C_{i_{N,k}}\big)=0_\aA\quad,\quad B\wedge\bigvee_{1\le k\le\gamma_N}\big(A_{N,k}\vee C_{i_{N,k}}\big)=0_\aA.\]
We have $\alpha(A\vee B)=\infty$. Let $x=(A,B)$, so that $x\in\xX\sm\{o\}$. For every $N\io$ and $1\le k\le\gamma_N$ set also
\[w_{N,k}=x\oplus\big(C_{i_{N,k+1}},0_\aA\big),\]
where we auxiliarily put
\[C_{i_{N,\gamma_N+1}}=C_{i_{N+1,1}}.\]
Note that for every $N\io$ and $1\le k\le\gamma_N$ we have $w_{N,k}\in\xX$ and
\[\tag{2}x,\, w_{N,k}\in\fF\big(x_{N,k},Z_{N,k}\big).\]

\medskip

Let now $G'$ be an $\Sacks$-generic filter over $V$ containing $s^{**}$ (and so $s$). Let us go to the extension $V[G']$. Let us again write $\trnorm{\,\cdot\,}=(\dot{n})_{G'}$, $m=(\dot{m})_{G'}$, $M=(\dot{M})_{G'}$, and $d=(\dot{d})_{G'}$. Note that by the assumption the norm $\trnorm{\,\cdot\,}$ is  SKK in $V[G']$. As in the proof of Theorem \ref{theorem:main_wMLUR}, using (1) and (2), we get a sequence $\seqN{k_N}$ of natural numbers such that 
\[\tag{3}\lim_{N\to\infty}\trnorm{\varphi\big(w_{N,k_N}\big)}=\trnorm{\varphi(x)}.\]
Of course, $\trnorm{\varphi(x)}\neq0$ and $\trnorm{\varphi\big(w_{N,k_N}\big)}\neq0$ for each $N\io$. Set
\[f=\varphi(x)/\trnorm{\varphi(x)}\quad\text{and}\quad f_N=\varphi\big(w_{N,k_N}\big)\big/\trnorm{\varphi\big(w_{N,k_N}\big)}\quad\text{for }N\io.\]

Let $\mu$ be a signed Radon measure on $St(\aA)$. Since the clopen sets $\clopen{C_{i_{N,k_N+1}}}$'s are pairwise disjoint, we have
\[\tag{4}\lim_{N\to\infty}\bigg|\int_{St(\aA)}\varphi(x)d\mu-\int_{St(\aA)}\varphi\big(w_{N,k_N}\big)d\mu\bigg|=\]
\[=\lim_{N\to\infty}\bigg|\int_{St(\aA)}\Big(\varphi(x)-\varphi\big(w_{N,k_N}\big)\Big)d\mu\bigg|\le\lim_{N\to\infty}|\mu|\bigg(\clopen{C_{i_{N,k_N+1}}}\bigg)=0.\]
Setting for each $N\io$
\[a_N=\Big(\trnorm{\varphi(x)}\cdot\trnorm{\varphi\big(w_{N,k_N}\big)}\Big)^{-1},\]
by (3) and (4), we get
\[\lim_{N\to\infty}\bigg|\int_{St(\aA)}\big(f-f_N\big)d\mu\bigg|=\]
\[=\lim_{N\to\infty}\bigg|\int_{St(\aA)}\Bigg(\frac{\varphi(x)}{\trnorm{\varphi(x)}}-\frac{\varphi\big(w_{N,k_N}\big)}{\trnorm{\varphi\big(w_{N,k_N}\big)}}\Bigg) d\mu\bigg|=\]
\[=\lim_{N\to\infty}a_N\cdot\bigg|\int_{St(\aA)}\Big(\trnorm{\varphi\big(w_{N,k_N}\big)}\cdot\varphi(x)-\trnorm{\varphi(x)}\cdot\varphi\big(w_{N,k_N}\big)\Big)d\mu\bigg|=\]
\[=\lim_{N\to\infty}a_N\cdot\bigg|\trnorm{\varphi\big(w_{N,k_N}\big)}\int_{St(\aA)}\varphi(x)d\mu-\trnorm{\varphi(x)}\int_{St(\aA)}\varphi\big(w_{N,k_N}\big)d\mu\bigg|=\]
\[=\trnorm{\varphi(x)}^{-2}\cdot\bigg|\trnorm{\varphi(x)}\int_{St(\aA)}\varphi(x)d\mu-\trnorm{\varphi(x)}\int_{St(\aA)}\varphi(x)d\mu\bigg|=0.\]
Consequently, as $\mu$ was arbitrary, the sequence $\seqN{f_N}$ converges weakly to $f$.

For each $N\io$ pick $t_N\in\clopen{C_{i_{N,k_N+1}}}$; note that $\varphi\big(w_{N,k_N}\big)\big(t_N\big)=1$ and $\varphi(x)\big(t_N\big)=0$ . Let $c>0$ be such that $\trnorm{\,\cdot\,}\ge c\|\cdot\|_\infty$. By a similar computation as above, we get
\[\limsup_{N\to\infty}\trnorm{f-f_N}=\]
\[=\limsup_{N\to\infty}\ a_N\cdot\trnorm{\vphantom{\Big|}\trnorm{\varphi\big(w_{N,k_N}\big)}\cdot\varphi(x)-\trnorm{\varphi(x)}\cdot\varphi\big(w_{N,k_N}\big)}\ge\]
\[\ge\limsup_{N\to\infty}\ a_N\cdot c\cdot\left\|\vphantom{\Big|}\trnorm{\varphi\big(w_{N,k_N}\big)}\cdot\varphi(x)-\trnorm{\varphi(x)}\cdot\varphi\big(w_{N,k_N}\big)\right\|_\infty\ge\]
\[\ge\limsup_{N\to\infty}\ a_N\cdot c\left|\vphantom{\Big|}\trnorm{\varphi\big(w_{N,k_N}\big)}\cdot\varphi(x)\big(t_N\big)-\trnorm{\varphi(x)}\cdot\varphi\big(w_{N,k_N}\big)\big(t_N\big)\right|=\]
\[=\limsup_{N\to\infty}\ a_N\cdot c\cdot\trnorm{\varphi(x)}=c/\trnorm{\varphi(x)}>0.\]
Consequently, the sequence $\seqN{f_N}$ does not converge to $f$ in the norm $\trnorm{\,\cdot\,}$, even though it converges to $f$ weakly, and so the norm $\trnorm{\,\cdot\,}$ is not SKK (in $V[G']$). A contradiction.

It follows that $C(St(\aA))$ does not admit any SKK renorming in $V[G]$.
\end{proof}

We immediately obtain the following important corollary.

\begin{corollary}\label{cor:main}
Assume that $V$ is a model of \textup{\textsf{ZFC}$+$\textsf{CH}} and let $\aA=\wo^V$. If $G$ is an $\Sacks$-generic filter over $V$, then in $V[G]$ it holds $\omega_1<\kappa\le\frakc$ and the Banach space $C(St(\aA))$ is a Grothendieck space of density $\omega_1$ which admits an equivalent strictly convex norm, but no wMLUR renormings and no SKK renormings.
\end{corollary}
\begin{proof}
    Let $G$ be an $\Sacks$-generic filter over $V$. By \cite[Theorem 1.14]{Bau85}, we have $\kappa\le\kappa^\omega=\frakc$ in $V[G]$. Moreover, since the forcing $\Sacks$ preserves $\omega_1$ (by \cite[Theorem 1.11]{Bau85}), we get that the Stone space $St(\aA)$ has weight $\omega_1$ and so the Banach space $C(St(\aA))$ has density $\omega_1$ in $V[G]$. By \cite[Theorem 3.1]{Bre06}, $C(St(\aA))$ is Grothendieck in $V[G]$. Also in $V[G]$, as the Stone space $St(\aA)$ is separable (having $\omega$ as a dense countable subspace), $C(St(\aA))$ isometrically embeds into $C(\bo)\cong\elli$, and since $\elli$ admits a strictly convex renorming (see \cite[Theorem 154]{GMZ22}), the space $C(St(\aA))$ admits such a renorming as well. On the other hand, that $C(St(\aA))$ does not admit any wMLUR renorming or any SKK renorming in $V[G]$ follows from Theorems \ref{theorem:main_wMLUR} and \ref{theorem:main_SKK}.
\end{proof}

For our last main result, we will need the following auxiliary notion from \cite{Ale01}.

\begin{definition}\label{def:property_cn}
    A Boolean algebra $\aA$ has the \textit{Countable Non-Completeness Property} (in short, \textit{(CNC)}) if for any countable antichain $\seqn{A_n}$ in $\aA$ there is $A\iA\sm\{0_\aA\}$ such that $A\wedge A_n=0_\aA$ for all $n\io$.
\end{definition}

It is easy to see that the Boolean algebra $\wo/Fin$ has property (CNC), whereas no $\sigma$-complete Boolean algebra has it. Consequently, $\wo/Fin$ satisfies the assumptions of the following theorem (cf. Corollary \ref{cor:main2}).

\begin{theorem}\label{theorem:main_SC}
Let $\aA\in V$ be a Boolean algebra which has properties (I) and (CNC) in $V$. Let $G$ be an $\Sacks$-generic filter over $V$. Then, in $V[G]$ the Banach space $C(St(\aA))$ admits no strictly convex renormings.
\end{theorem}
\begin{proof}
The proof is very much similar to the one of Theorem \ref{theorem:main_SKK} (and so of Theorem \ref{theorem:main_wMLUR}). Thus, starting in the ground model $V$, let the objects as in Section \ref{sec:aux_ba} for $\aA$ be given. Then again, in $V[G]$, assume that an equivalent norm $\trnorm{\,\cdot\,}$ on $C(St(\aA))$ is strictly convex. Back in $V$, let $\dot{n},\dot{m},\dot{M},\dot{d}$ be $\Sacks$-names such that $(\dot{n})_G=\trnorm{\,\cdot\,}$, $(\dot{m})_G=m$, $(\dot{M})_G=M$, and $(\dot{d})_G=d$, and fix a condition $s\in G$ forcing that $\dot{n}$ is an equivalent strictly convex norm on $C(St(\check{\aA}))$ and that $(\dot{m},\dot{M},\dot{d})$ is the triple associated with $\dot{n}$. Let sequences
\begin{itemize}
    \item $\Big\langle\big\langle A_{N,1},\ \ldots,\ A_{N,\gamma_N}\big\rangle\colon\ N\io\Big\rangle$,
    \item $\Big\langle\big\langle B_{N,1},\ \ldots,\ B_{N,\gamma_N}\big\rangle\colon\ N\io\Big\rangle$,
    \item $\Big\langle\big\langle i_{N,1},\ \ldots,\ i_{N,\gamma_N}\big\rangle\colon\ N\io\Big\rangle$,
    \item $\Big\langle\big\langle x_{N,1},\ \ldots,\ x_{N,\gamma_N}\big\rangle\colon\ N\io\Big\rangle$,
    \item $\Big\langle\big\langle Z_{N,1},\ \ldots,\ Z_{N,\gamma_N}\big\rangle\colon\ N\io\Big\rangle$,
    \end{itemize}
a pair $x\in\xX\sm\{o\}$, and a forcing condition $s^{**}\in\Sacks$ (with $s^{**}\le s$) be exactly as in the proof of Theorem \ref{theorem:main_SKK}. Since the Boolean algebra $\aA$ has property (CNC), there is $C\in\aA\sm\{0_\aA\}$ such that
\[C\wedge\bigvee_{1\le k\le\gamma_N}\big(A_{N,k}\vee B_{N,k}\vee C_{i_{N,k}}\big)=0_\aA\]
for every $N\io$. Let
\[y_+=x\oplus\big(C,0_\aA\big)\quad\text{and}\quad y_-=x\oplus\big(0_\aA,C\big).\]
Again for every $N\io$ and $1\le k\le\gamma_N$ we have
\[\tag{1}x,\ y_\pm\in\fF\big(x_{N,k},Z_{N,k}\big).\]

\medskip

Let now $G'$ be an $\Sacks$-generic filter over $V$ containing $s^{**}$ (and so $s$). Let us go to the extension $V[G']$. Let us again also write $\trnorm{\,\cdot\,}=(\dot{n})_{G'}$, $m=(\dot{m})_{G'}$, $M=(\dot{M})_{G'}$, and $d=(\dot{d})_{G'}$. Note that by the assumption the norm $\trnorm{\,\cdot\,}$ is  strictly convex in $V[G']$. As in the proofs of Theorems \ref{theorem:main_wMLUR} and \ref{theorem:main_SKK} (together with the aid of (1)), we get a sequence $\seqN{k_N}$ of natural numbers such that for every $N\io$ we have
\[m\big(o,\, x_{N,k_N},\, Z_{N,k_N}\big)\le\trnorm{\varphi(x)}\le M\big(o,\, x_{N,k_N},\, Z_{N,k_N}\big)\]
and
\[m\big(o,\, x_{N,k_N},\, Z_{N,k_N}\big)\le\trnorm{\varphi(y_\pm)}\le M\big(o,\, x_{N,k_N},\, Z_{N,k_N}\big).\]
By the limit argument, it again follows that
\[\trnorm{\varphi(x)}=\trnorm{\varphi(y_\pm)}.\]
Note that $\trnorm{\varphi(x)}>0$. Let
\[f=\varphi(x)/\trnorm{\varphi(x)}\quad\text{and}\quad f_\pm=\varphi(y_\pm)/\trnorm{\varphi(y_\pm)}.\]
We have $\trnorm{f}=\trnorm{f_\pm}=1$ and $f=(f_++f_-)/2$, and so $\trnorm{(f_++f_-)/2}=1$, yet $f_+\neq f_-$. Consequently, the norm $\trnorm{\,\cdot\,}$ is not strictly convex in $V[G']$, which is a contradiction.

So, it follows that $C(St(\aA))$ does not admit any strictly convex renormings in $V[G]$.
\end{proof}

Note that if a Boolean algebra $\aA$ is such that $C(St(\aA))$ is Grothendieck and $\bB$ is a quotient Boolean algebra of $\aA$, then $C(St(\bB))$ is Grothendieck as well (cf. \cite[Proposition 2.11.(b)]{Sch82}). Also, as said, the algebra $\wo/Fin$ has both properties (I) and (CNC). Consequently, as in the proof of Corollary \ref{cor:main} and by Theorem \ref{theorem:main_SC}, we get the following result, which immediately implies Main Theorem from Introduction.

\begin{corollary}\label{cor:main2}
Assume that $V$ is a model of \textup{\textsf{ZFC}$+$\textsf{CH}} and let $\aA=(\wo/Fin)^V$. If $G$ is an $\Sacks$-generic filter over $V$, then in $V[G]$ it holds $\omega_1<\kappa\le\frakc$ and the Banach space $C(St(\aA))$ is a Grothendieck space of density $\omega_1$ which has no strictly convex or SKK renormings.
\end{corollary}

\begin{remark}
    Note that by the main result of \cite{SZNik}, the Boolean algebra $\aA$ from Corollary \ref{cor:main2} (and so the algebra of all clopen subsets of $St(\aA)$) has also the so-called \textit{Nikodym property} in $V[G]$, that is, every sequence $\seqn{\mu_n}$ in $V[G]$ of finite signed finitely additive measures on $\aA$ for which we have $\sup_{n\io}\big|\mu_n(A)\big|<\infty$ for all $A\in\aA$ is norm bounded, i.e. $\sup_{n\io}\big\|\mu_n\big\|<\infty$. In particular, it follows that the linear subspace $\spn\!\big\{\chi_{\clopen{A}}\colon A\in\aA\big\}$ of $C(St(\aA))$ is \textit{barrelled} in the supremum norm topology (see \cite[Definition 2.4]{Sch82}).
\end{remark}

\section{Concluding remarks}

In the final section of the paper we provide several further remarks as well as pose some open questions. 

\subsection{Weaker completeness and interpolation properties}


Consider the following two properties of Boolean algebras weaker than the $\sigma$-completeness, introduced by Haydon \cite{Hay81} and Freniche \cite{Fre84_vhs}, respectively.

\begin{definition}
    A Boolean algebra $\aA$ has the \textit{Subsequential Completeness Property} (in short, \textit{(SC)}) if for every countable antichain $\seqn{A_n}$ in $\aA$ there is $M\in\cso$ such that the supremum $\bigvee_{n\in M}A_n$ exists in $\aA$.
\end{definition}

\begin{definition}
    A Boolean algebra $\aA$ has the \textit{Subsequential Interpolation Property} (in short, \textit{(SI)}) if for every countable antichain $\seqn{A_n}$ in $\aA$ and every $M\in\cso$ there are $A\iA$ and $M'\in\ctblsub{M}$ such that $A_n\le A$ for every $n\in M'$ and $A_n\wedge A=0_\aA$ for every $n\in\omega\sm M'$.
\end{definition}
\noindent We have the following obvious diagram of implications:
\vspace{-1cm}$$\xymatrixcolsep{3pc}\xymatrix{
\\
\sigma\text{-compl.}\ar@{=>}[r]\ar@{=>}[d]&    \text{(I)}\ar@{=>}[d]\\
\text{(SC)}\ar@{=>}[r]&     \text{(SI)}\\
}$$
None of the implications can be reversed, see \cite{Fre84_vhs,Hay81,See68}.

Alexandrov \cite{Ale01} showed that if a Boolean algebra has property (SI), then the Banach space $C(St(\aA))$ does not admit any renorming which is wMLUR or SKK. It seems that combining his methods together with our arguments presented in Section \ref{sec:main}, one can weaken the assumptions in the main theorems of that section and require only that ground model Boolean algebras $\aA$ have property (SI) instead of property (I). The crucial change of the argument would require employing Lemma \ref{lemma:basic}\ref{lemma:basic:exists} in its full generality, that is, for sets $\yY$ bigger than $\{o\}$, in order to control elements $A,B,C\in\aA$, given by property (SI) and covering subantichains of the inductively constructed antichains, cf. \cite{Ale01} for details. However, we deliberately refrained from extending our results to the case of property (SI), since we believe that such an extension would be mostly of technical character rather than of qualitative one.

\medskip

For various other interpolation and completeness properties of Boolean algebras, see e.g. \cite[Section 19.5]{KKLPS}, \cite{KS13,Sob20}, and references therein.

\subsection{Cardinal characteristics of the continuum}

The results presented in Section \ref{sec:main} suggest introducing the following cardinal numbers:
\[\frakwmlur=\min\big\{w(K)\colon\ C(K)\text{ admits no wMLUR renorming}\big\},\]
\[\frakskk=\min\big\{w(K)\colon\ C(K)\text{ admits no SKK renorming}\big\},\]
\[\fraksc=\min\big\{w(K)\colon\ C(K)\text{ admits no strictly convex renorming}\big\}.\]
Of course, one can introduce similar cardinal numbers for other properties of norms defined in Section \ref{sec:intro}.

It is immediate that $\frakwmlur\le\fraksc$. Since any separable Banach space admits a LUR renorming (see \cite{GMZ22}) as well as the space $C(\omega^*)\cong\elli/c_0$ admits no strictly convex renormings nor any SKK renormings and has density $\frakc$, we also have $\omega_1\le\frakwmlur\le\fraksc\le\frakc$ and $\omega_1\le\frakskk\le\frakc$. Further, Corollaries \ref{cor:main} and \ref{cor:main2} imply the following consistency result.

\begin{corollary}
    It is consistent with \textup{\textsf{ZFC}} that $\frakwmlur=\frakskk=\fraksc=\omega_1<\frakc$.
\end{corollary}

We are however not aware of any model of \textsf{ZFC} where at least some of the above-defined cardinal numbers are strictly greater than $\omega_1$.

\begin{question}\label{ques:greater_than_omega1}
    Let $\frakx\in\{\frakwmlur,\frakskk,\fraksc\}$. Is there a model of \textup{\textsf{ZFC}} in which $\frakx>\omega_1$?
\end{question}

A natural place to consider the above question is any model of Martin's axiom and the negation of the Continuum Hypothesis.

\begin{question}\label{ques:ma}
    Assume \textup{\textsf{MA}}. Is it true that $\frakwmlur=\frakskk=\fraksc=\frakc$? 
\end{question}

A typical approach to both Questions \ref{ques:greater_than_omega1} and \ref{ques:ma} could be to prove that in \textsf{ZFC} any of the numbers $\frakwmlur$, $\frakskk$, or $\fraksc$ is greater than some of the standard cardinal characteristics occurring e.g. in Cicho\'n's diagram or van Douwen's diagram (see \cite{Bla10}). However, as the next proposition shows this is pointless, since both the additivity of the null ideal, $\add(\nN)$, and the pseudointersection number $\frakp$, that is, the smallest characteristics in Cicho\'n's diagram and van Douwen's diagram, respectively, can be consistently strictly larger than $\frakwmlur$ and $\fraksc$. As a model of \textsf{ZFC} witnessing this one can take any model of \textsf{PFA}$(S)$, the Proper Forcing Axiom for proper posets not destroying a given fixed \textit{Suslin tree} $S$ (i.e. a tree of height $\omega_1$ without uncountable chains and uncountable antichains).

\begin{proposition}\label{prop:suslin_characteristics}
It is consistent with \textup{\textsf{ZFC}} that
\[\frakwmlur=\fraksc=\omega_1<\frakp=\add(\nN)=\omega_2=\frakc.\]
\end{proposition}
\begin{proof}
    Let $W$ be a model of \textsf{ZFC}+\textsf{PFA}$(S)$, where $S$ is a Suslin tree. By \cite[Proposition 2.1]{RY12}, it holds $\omega_1<\frakp=\add(\nN)=\omega_2=\frakc$ in $W$. As $S$ is a Suslin tree in $W$, it induces a compact Suslin line $K_S$ (i.e. a Suslin line with both end-points added; see \cite[Theorem 5.13]{Kun80}). By \cite[Example 3]{HJNR00}, the Banach space $C(K_S)$ admits no strictly convex renorming. Moreover, as $w(K_S)=\omega_1$ and so $d(C(K_S))=\omega_1$, we have $\frakwmlur=\fraksc=\omega_1$ in $W$.
\end{proof}

However, a proper \textsf{ZFC} proof of any inequality like, e.g., $\fraksc\le\frakp$ would be still desired. 

\begin{question}
     Let $\frakx\in\{\frakwmlur,\frakskk,\fraksc\}$. Is there any standard cardinal characteristic of the continuum $\fraky$, e.g. in Cicho\'n's diagram or van Douwen's diagram, for which we have $\frakx\le\fraky$ in \textup{\textsf{ZFC}}?
\end{question}

\begin{remark}\label{remark:suslin_sbs_Sacks}
    By \cite[Theorem A]{HJNR00}, the space $K_S$ (induced by the Suslin tree $S$) from the proof of the latter proposition admits a (sequential) Kadets--Klee renorming. Consequently, it does not witness that $\frakskk=\omega_1$ holds in the model $W$ (in fact, we do not know the value of $\frakskk$ in $W$). For this kind of reason, to establish Corollary \ref{cor:main2}, we could not exploit any Suslin tree existing in the extension $V^{\Sacks}$ (note that Suslin trees do exist in $V^{\Sacks}$, because the Diamond Principle $\Diamond(\non(\mM))$ holds there, see \cite{MHD03}), but we had to use some other structure, e.g. the ground model Boolean algebra $(\wo/Fin)^V$. Also, the space $C(K_S)$ is not Grothendieck, as $K_S$ contains non-trivial convergent sequences.
\end{remark}

To complete Proposition \ref{prop:suslin_characteristics}, let us also pose the following question concerning $\frakskk$.

\begin{question}
    What is the value of $\frakskk$ in models of \textup{\textsf{PFA}}$(S)$?
\end{question}

Our last question asks whether the numbers $\frakwmlur$, $\fraksc$, and $\frakskk$ can be distinguished.

\begin{question}
    Is any of the following inequalities consistently possible:
    \begin{itemize}
        \item $\frakwmlur<\fraksc$,
        \item $\frakwmlur\neq\frakskk$,
        \item $\fraksc\neq\frakskk$?
    \end{itemize}
\end{question}

Of course, the latter question has only sense in the case when Question \ref{ques:greater_than_omega1} has an affirmative answer.

\subsection{Grothendieck $C(K)$-spaces}

The Banach spaces $C(K)$ consistently constructed in Corollaries \ref{cor:main} and \ref{cor:main2} are Grothendieck spaces which admit neither wMLUR renormings nor SKK renormings, so in particular no MLUR, wLUR, LUR, or KK renormings. As it is well known, the same regards the spaces $C(\bo)\cong\elli$ and $C(\omega^*)\cong\elli/c_0$ (see \cite{Ale01} or \cite{GMZ22}), which are classical examples of Grothendieck Banach spaces $C(K)$, as well as Haydon's Grothendieck space $C(K)$ from \cite{Hay81} which does not contain any isomorphic copy of $\elli$ (see \cite{AB88}). It seems that no example of a Grothendieck $C(K)$-space admitting, e.g., a LUR renorming is known (cf. Questions in \cite{Ale01} and \cite{AB88}).

\begin{question}
    Does there exist a Grothendieck $C(K)$-space which admits a (w)(M)LUR renorming or a (S)KK renorming?
\end{question}

Concerning strictly convex renormings, recall that $C(\bo)$ admits such a renorming whereas $C(\beta\omega_1)\cong\elli(\omega_1)$ does not (see \cite{GMZ22}).

\subsection{Adding other reals\label{sec:adding}}

An intuitive meaning of the results presented in Section \ref{sec:main} is that the side-by-side Sacks forcing $\Sacks$ does not add any new renormings of the Banach spaces $C(St(\aA))$ for ground model Boolean algebras $\aA$ with properties (I) and (CNC) which are wMLUR, SKK, or strictly convex. Note that the arguments too work without any changes for $\kappa=1$ and $\Sacks=\SS$ (so for the ordinary Sacks forcing). Moreover, it seems that exploiting the methods presented in \cite{SZForExt} one could prove the results also for other classes of forcings such as, e.g., the Miller forcing or the Silver forcing. However, at this moment, it would be rather more interesting to find an example of a standard notion of forcing which \emph{does add} a new renorming which is, e.g., LUR. 

\begin{question}\label{ques:forcing_renorming}
    Which standard notions of forcing $\PP$ have the following property: for every ground model $\sigma$-complete Boolean algebra $\aA\in V$ and every $\PP$-generic filter $G$ over $V$, in $V[G]$ the Banach space $C(St(\aA))$ admits a renorming which is wMLUR, SKK, or strictly convex?
\end{question}

Note that a trivial and uninteresting way to force a LUR renorming of the Banach space $C(St(\aA))$ for any ground model Boolean algebra of size $\le\frakc$ is by collapsing $\frakc$ to $\omega$ (see, e.g., \cite[Example 15.20]{Jec03}), thus making the Stone space $St(\aA)$ metrizable and so the Banach space $C(St(\aA))$ separable.

In \cite{SZAdding} it was showed that extending the ground model $V$ by adding a single Cohen or random real also adds new sequences of functionals witnessing that the Banach spaces $C(St(\aA))$ for all ground model $\sigma$-complete Boolean algebras $\aA$ do not have the Grothendieck property, even though they do have it in $V$. Inspired by this, we pose the following variant of Question \ref{ques:forcing_renorming}.

\begin{question}
    Let $\PP\in V$ be a notion of forcing adding a Cohen real or a random real. Let $\aA\in V$ be a $\sigma$-complete Boolean algebra. Assume that $G$ is a $\PP$-generic filter over $V$. In $V[G]$, does the Banach space $C(St(\aA))$ admit a renorming which is, e.g., LUR?
\end{question}

\section*{AI statement}

No artificial intelligence tools have been used at any stage of conducting the presented research or during the preparation and edition of the paper.

\end{document}